\documentclass[12pt]{amsart}
\usepackage{amsmath,amssymb}
\usepackage{amsfonts}
\usepackage{amsthm}
\usepackage{latexsym}
\usepackage{graphicx}

\def\R{\mathbb{R}}

\def\vv<#1>{\langle#1\rangle}
\def\ol{\overline}

\def\XXint#1#2{\setbox0=\hbox{$#1{#2}{\int}$}{#2}\kern-.5\wd0 }

\def\XXint#1#2#3{{\setbox0=\hbox{$#1{#2#3}{\int}$}
     \vcenter{\hbox{$#2#3$}}\kern-.5\wd0}}

\def\vv<#1>{{\left\langle#1\right\rangle}}

\def\Int{{\rm Int}}
\newtheorem{thm}{Theorem}[section]

\newtheorem{lem}{Lemma}[section]

\theoremstyle{definition}

\theoremstyle{remark}

\numberwithin{equation}{section}

\begin{document}
\title{Smooth compositions with a nonsmooth inner function}

\author{Yongjie Shi}
\address{Department of Mathematics, Shantou University, Shantou, Guangdong, 515063, China}
\email{yjshi@stu.edu.cn}
\author{Chengjie Yu$^1$}
\address{Department of Mathematics, Shantou University, Shantou, Guangdong, 515063, China}
\email{cjyu@stu.edu.cn}
\thanks{$^1$Research partially supported by the Yangfan project from Guangdong Province and NSFC 11571215.}
\renewcommand{\subjclassname}{%
  \textup{2010} Mathematics Subject Classification}
\subjclass[2010]{Primary 53E52; Secondary 54C10}
\date{}
\keywords{Baire space, Baire's category theorem, periodic function}
\begin{abstract}
In this paper, we present an interesting application of Baire's category theorem.
\end{abstract}
\maketitle\markboth{Shi \& Yu}{Smooth compositions with a nonsmooth inner function}
\section{Introduction}
In this paper, we prove the following interesting result:
\begin{thm}\label{thm-main}
Let $n$ be a nonnegative integer or $\infty$, $p:\R\to\R$ be a given function with $p\not\in C^n(\R,\R)$ and
\begin{equation}
\mathcal{A}_p=\{f\in C^n(\R,\R)\ |\ f(p(\cdot)+c)\in C^n(\R,\R) \mbox{ for any }c\in\R\}.
\end{equation}
Then, either $\mathcal A_p=\R$ or $\mathcal A_p=C_d^n(\R,\R)$ for some nonzero constant $d$. Here $C_d^n(\R,\R)$ means the collection of all functions in $C^n(\R,\R)$ with periodicity $d$.
\end{thm}
This result is motivated by the work \cite{CW} of   Christensen and Wu on diffeological vector spaces. A weaker form of Theorem \ref{thm-main} is needed in \cite{CW}. The method in \cite{CW} dealing with the weaker form seems not applicable to prove Theorem \ref{thm-main}.

It is clear that $\mathcal A_p$ is a translating invariant subalgebra of $C^n(\R,\R)$. As an example, take $p=\chi_E$ with $E$ an abitrary subset of $\R$ ($E\not=\emptyset,\R$). Then, it is  clear that $\cos(2\pi x), \sin(2\pi x)\in \mathcal A_p$. In fact, it is not hard to see that $\mathcal A_p=C^n_1(\R,\R)$ in this specific example.

A key ingredient in the proof of Theorem \ref{thm-main} is the following result about the continuity of maps with $\sigma$-compact graph.
\begin{lem}\label{lem-closed-map}
Let $X$ be a Hausdorff Baire space, $Y$ be a topological space and $f:X\to Y$ be a map with $\sigma$-compact graph. Then, $f$ is continuous on a dense open subset of $X$.
\end{lem}
For a Baire space, we mean a topological space satisfying Baire's category theorem. That is, a countable intersection of open dense subset is still dense.
  Typical examples of Baire spaces are complete metric spaces and locally compact Hausdorff spaces. One may compare Lemma \ref{lem-closed-map} with Blumberg's theorem \cite{Bl} : for any real function on $\R$, there is a dense subset $E\subset \R$, such that $f|_E:E\to\R$ is continuous. Blumberg's theorem was later generalized by Bradford and Goffman \cite{BG} for more general spaces (see also \cite{Ho}).

 \noindent{\bf Acknowledgement.} The authors would like to thank the referee for generously sharing ideas that significantly strengthen the results (both Theorem \ref{thm-main} and Lemma \ref{lem-closed-map}) of the previous version of this paper, and to thank their colleague, Prof. Enxin Wu, for many helpful discussions.
\section{Proof of Theorem \ref{thm-main}}
We first prove Lemma \ref{lem-closed-map}.
\begin{proof}[Proof of Lemma \ref{lem-closed-map}] Let
\begin{equation}
\Gamma=\{(x,f(x))\in X\times Y\ |\ x\in X\}
\end{equation}
be the graph of $f$. Suppose that $\Gamma=\cup_{n=1}^\infty K_n$ where $K_n$ is
a compact subset of $X\times Y$ for $n=1,2,\cdots$. Then, $A_n=\pi_X (K_n)$ is a compact subset of $X$, where $\pi_X:X\times Y\to X$ is the natural projection. Note that $f|_{A_n}:A_n\to Y$ as a map with compact graph is continuous. To see this, let $F$ be any closed subset of $Y$, then
\begin{equation}
(f|_{A_n})^{-1}(F)=\pi_{X}(K_n\cap (X\times F))
\end{equation}
is compact and hence closed.

Let $F_n=A_n\setminus {\rm Int}(A_n)$ where ${\rm Int}(A_n)$ is the  interior of $A_n$. Then, $U_n=X\setminus F_n$ is a dense open subset of $X$. So, $\cap_{n=1}^\infty U_n$ is a dense subset of $X$ by that $X$ is a Baire space. On the other hand, since $X=\cup_{n=1}^\infty A_n$,
\begin{equation}
\cap_{n=1}^\infty U_n\subset \cup_{n=1}^\infty \Int(A_n).
\end{equation}
Hence $\cup_{n=1}^\infty \Int(A_n)$ is a dense open subset of $X$. Moreover, $f$ is continuous at the points in $\cup_{n=1}^\infty \Int(A_n)$ since $f|_{A_n}$ is continuous for $n=1,2,\cdots$. This completes the proof of the Lemma.
\end{proof}
For clarity, we will separate the proof of Theorem \ref{thm-main} into two cases:(i) graph of $p$ is closed and (ii) graph of $p$ is not closed.
\begin{lem}\label{lem-closed}
Let notations be the same as in Theorem \ref{thm-main}. Suppose that the graph of $p$ is closed, then $\mathcal A_p=\R$.
\end{lem}
\begin{proof}
We will proceed by contradiction. Assume that $\mathcal A_p\neq\R$. Let $G\subset \R$ be the largest open subset of $\R$ such that $p$ is continuous on $G$. By Lemma \ref{lem-closed-map}, $G$ is dense in $\R$. Let $F=\R\setminus G$.
For clarity, we divide the proof into several claims.\\
\noindent {\bf Claim 1.}\emph{By the structure of open subsets in $\R$, $G$ is a disjoint union of countably many open intervals. Let $(a,b)$ be one of such open intervals where $a$ may be $-\infty$ and $b$ may be $+\infty$. Then, $p\in C(\ol{(a,b)})$.}\\
{\bf Proof of Claim 1.} We only show the case that $a=-\infty$ and $b$ is finite. The proofs of the other cases are similar. In this case, $\ol{(a,b)}=(-\infty,b]$.

Because the graph of $p$ is closed and $p\in C((-\infty,b)),$  the asymptotic behavior of $p$ as $x$ approaching $b^{-}$ happens in only three cases:
\begin{enumerate}
\item  $\lim_{x\to b^-}p(x)=p(b)$,
\item $\lim_{x\to b^-}p(x)=+\infty$, and
\item  $\lim_{x\to b^-}p(x)=-\infty$.
\end{enumerate}
To prove { Claim 1}, we only need to exclude (2) and (3). If (2) is true, we have
\begin{equation}
\lim_{y\to+\infty}f(y)=\lim_{x\to b^-}f(p(x)+a)=f(p(b)+a)
\end{equation}
for any $a\in \R$ and $f\in \mathcal A_p$. This means that $\mathcal A_p=\R$ which is a contradiction. Case (3) can be excluded similarly.\\
{\bf Claim 2.} \emph{$F$ has no isolated points.}\\
{\bf Proof of Claim 2.} Suppose that $x_0$ is an isolated point of $F$. Then, by {\bf Claim 1}, $p$ is continuous on a neighborhood of $x_0$. So, $x_0\in G$ which is a contradiction.\\
{\bf Claim 3.} \emph{Any continuous point of $p|_F:F\to \R$ is a continuous point of $p$.}\\
{\bf Proof of Claim 3.} Let $x_0\in F$ be a continuous point of $p|_F$. If $x_0$ is not a continuous point of $p$.  Then there is a $\epsilon_0>0$ and a sequence $\{x_1,x_2,\cdots,x_n,\cdots\}$ of points tending to $x_0$ as $n\to\infty$, such that
\begin{equation}\label{eqn-out}
|p(x_n)-p(x_0)|\geq \epsilon_0
\end{equation}
for any $n=1,2,\cdots$. Since $x_0$ is a continuous point of $p|_{F}$, $x_n\not \in F$ for $n$ large enough. Moreover, since $x_0$ is not isolated in $F$ (by Claim 2) and by Claim 1,
$x_n\in (a_n,b_n)$ where $(a_n, b_n)$ is an open interval in the disjoint open interval decomposition of $G$ for $n$ large enough. It is clear that $a_n\to x_0$ and $b_n\to x_0$ as $n\to\infty$. Furthermore, since $x_0$ is a continuous point of $p|_{F}$ and $a_n\in F$,
\begin{equation}\label{eqn-in}
|p(a_n)-p(x_0)|<\epsilon_0/2
\end{equation}
for $n$ large enough. Since $p\in C([a_n,b_n])$ (by Claim 1), and by \eqref{eqn-out} and \eqref{eqn-in}, there is a point $\tilde x_n\in [a_n,b_n]$ such that
\begin{equation}
p(\tilde x_n)=p(x_0)+\epsilon_0/2
\end{equation}
or
\begin{equation}
p(\tilde x_n)=p(x_0)-\epsilon_0/2.
\end{equation}
Then there is a sequence $\{\tilde x_1,\tilde x_2,\cdots\}$ of points  tending to $x_0$ as $n\to \infty$ (taking subsequence if necessary) such that
\begin{equation}
p(\tilde x_n)\to p(x_0)+\epsilon_0/2
\end{equation}
or
\begin{equation}
p(\tilde x_n)\to p(x_0)-\epsilon_0/2
\end{equation}
as $n\to\infty$.
This is a contradiction of that the graph of $p$ is closed
{\bf Claim 4.} \emph{$p$ is continuous on $\R$.}\\
{\bf Proof of Claim 4.} Assume that $F\neq\emptyset$. By  Lemma \ref{lem-closed-map},  $p|_F:F\to \R$ is continuous on some dense open subset of $F$. So, there is an open interval $I$ such that $p|_F$ is continuous on $I\cap F$ with $I\cap F\neq\emptyset$. By Claim 3, points in $I\cap F$ are continuous points of $p$. So $p$ is continuous on $I$ which implies that $I\subset G$ and contradicts that $I\cap F\neq\emptyset$.

We are now ready to complete the proof. Since $p\not\in C^n(\R,\R)$, we know that $n\geq 1$. Let $f\in \mathcal A_p$ be a nonconstant function. Then, $f'(y_0)\neq 0$ for some $y_0\in \R$. It is clear that
\begin{equation}
\tilde f(y)=f(y-p(x_0)+y_0)
\end{equation}
also belongs to $\mathcal A_p$ and
\begin{equation}
\tilde f'(p(x_0))=f'(y_0)\neq 0.
\end{equation}
Let $g(x)=\tilde f(p(x))$. Then $g\in C^n(\R,\R)$ because $\tilde f\in \mathcal A_p$. Since $p$ is continuous at $x_0$,
\begin{equation}
p(x)=\tilde f^{-1}(g(x))
\end{equation}
for any $x$ in some neighborhood of $x_0$. This implies that $p$ is $C^n$ in some neighborhood of $x_0$. Because $x_0$ is arbitrary, $p\in C^n(\R,\R)$ which is a contradiction.
\end{proof}
\begin{lem}\label{lem-non-closed}
Let notations be the same as in Theorem \ref{thm-main}. If the graph of $p$ is not closed, then either $\mathcal A_p=\R$ or $\mathcal A_p=C_d^n(\R,\R)$.
\end{lem}
\begin{proof}
Let
\begin{equation}
L=\{d\in \R\ |\ f(y+d)=f(y)\ \mbox{for any $y\in \R$ and any $f\in \mathcal A_p$}\}.
\end{equation}
It is clear that $L$ is closed subgroup of $\R$ since $\mathcal A_p\subset C(\R,\R)$. Let
\begin{equation}
\Lambda(x)=\{y\in \R\ |\ y=\lim_{n\to\infty}p(x_n)\ \mbox{for some sequence}\ x_n\to x\}.
\end{equation}
{\bf Claim 1.}\emph{For any $y\in \Lambda(x)$, $y-p(x)\in L$.}\\
{\bf Proof of Claim 1.} Let $\{x_n\}$ be a sequence of points tending to $x$ such that
\begin{equation}
p(x_n)\to y.
\end{equation}
Then, for any $f\in \mathcal A_p$,
\begin{equation}
f(p(x)+c)=\lim_{n\to\infty}f(p(x_n)+c)=f(y+c)
\end{equation}
for any $c\in \R$ because $f(p(\cdot)+c)$ is continuous. This means that $y-f(x)\in L$.

Because the graph of $p$ is not closed, there is a point $x\in \R$, such that $\Lambda(x)\varsupsetneq \{p(x)\}$. Then, by Claim 1, we know that $L\neq \{0\}$. Therefore, $L=\R$ or $L=d\mathbb Z$ for some nonzero constant $d$. For the first case, we know that $\mathcal A_p=\R$. For the second case, we have $\mathcal A_p\subset C_d^n(\R,\R)$.

Let $\pi:\R\to\R/d\mathbb Z$ be the natural projection. Then, by Claim 1, it is clear that $\pi(p+c):\R\to \R/d\mathbb Z$ is continuous. Moreover, for any $f\in C_d^n(\R,\R)$, $f$ descends to a function $\ol f:\R/d\mathbb Z\to \R$ such that
\begin{equation}
f=\ol f\circ\pi.
\end{equation}
It is clear that $\ol f\in C^n(\R/d\mathbb Z,\R)$.

When $n=0$, for any $f\in C_d^0(\R,\R)$, we know that
\begin{equation}
f(p+c)=\ol f\circ\pi(p+c)\in C^0(\R,\R).
\end{equation}
This means that $f\in \mathcal A_p$. So, we have shown that $\mathcal A_p=C_d^n(\R,\R)$ for the case $n=0$.

When $n\geq 1$, if $\mathcal A_p\neq\R$, let $f_0$ be a nonconstant function in $\mathcal A_p$. Then a similar argument as in the proof of Lemma \ref{lem-closed} (after Claim 4 in the proof of Lemma \ref{lem-closed}) using inverse function theorem will show that $\pi(p+c)\in C^n(\R/d\mathbb Z,\R)$. Then, for any $f\in C_d^n(\R,\R)$,
\begin{equation}
f=\ol f\circ\pi(p+c)\in C^n(\R,\R).
\end{equation}
This means that $f\in \mathcal A_p$. So $\mathcal A_p=C_d^n(\R,\R)$.
\end{proof}
\begin{proof}[Proof of Theorem \ref{thm-main}] The combination of Lemma \ref{lem-closed} and Lemma \ref{lem-non-closed} gives us Theorem \ref{thm-main}.
\end{proof}

\end{document}